\documentclass[graybox]{svmult}

\usepackage{amsfonts,amstext,amssymb,physics} 
\usepackage{amsmath}

\usepackage{amssymb}
\usepackage{mathptmx}       
\usepackage{helvet}         
\usepackage{courier}        
\usepackage{type1cm}        
\usepackage{graphicx}        
\usepackage{multicol}        
\usepackage[bottom]{footmisc}

\usepackage{enumerate}
\usepackage{tikz}

\usepackage{enumitem}  
\usepackage{calc}
\setlist{labelindent=1pt,itemsep=0.1cm}
\setlist[itemize]{leftmargin=0.7cm}
\setlist[enumerate]{itemindent=0em,leftmargin=0.7cm}

\usepackage{float} 

\usepackage[bookmarks=false,hidelinks,unicode]{hyperref}

\smartqed

\allowdisplaybreaks

\begin{document}
\title*{Fixed point results for set-contractions on semi-metric space with a directed graph}
\titlerunning{Fixed point results for set-contractions on semi-metric space}
\author{Talat Nazir  \and Zakaria Ali \and Shahin Nosrat Jogan \and Sergei Silvestrov}
\authorrunning{T. Nazir, Z. Ali, S. N. Jogan, S. Silvestrov}

\institute{Talat Nazir, Zakaria Ali, Shahin Nosrat Jogan
\at
Department of Mathematical Sciences, University of South Africa, Florida 0003, South Africa. \\  \email{talatn@unisa.ac.za, alizi@unisa.ac.za, nosgo.bots@gmail.com}
\and
Sergei Silvestrov
\at Division of Mathematics and Physics, School of Education, Culture and Communication, M{\"a}lardalen University, Box 883, 72123 V{\"a}ster{\aa}s, Sweden. \\ 
\email{sergei.silvestrov@mdu.se}
}
%
%


\maketitle
\label{chap:NazirAliJoganSilvestrov:FPSMS}

\abstract*{Fixed point results with respect to generalized rational contractive mappings in semi-metric spaces endowed with a directed graph are proved. Some examples are provided to illustrate the results. The obtained results extend, improve and generalize many results in the existing literature.
}

\abstract{Fixed point results with respect to generalized rational contractive mappings in semi-metric spaces endowed with a directed graph are proved. Some examples are provided to illustrate the results. The obtained results extend, improve and generalize many results in the existing literature.
}

\section{Introduction}
\label{Section:section-1-introduction-semi-metric}
Among various different kinds of extensions or restrictions of the definition of a metric  one of them which we consider in this chapter is the Hausdorff semi-metric $\mathrm{d}_{s}$ imposed on a nonempty set $Y$. The main restriction of such a metric is that it does not satisfy the \emph{triangle inequality}. We prove our results by employing the generalized rational graph map on the closed and bounded subsets of the semi-metric space $(Y, \mathrm{d}_{s})$.
Frechet \cite{Frechet1906} studied
the concept of metric space defined on a nonempty set X that induces the Hausdorff topology on $X$. Consequently, there have been several generalizations of a metric function. Among which, one of the important generalization is the concept of a semi-metric or symmetric metric space which gives rise to a topology which is not a Hausdorff.
Zhu et al. \cite{Zhu2005} studied the equivalent contractive conditions in symmetric spaces. Imdad \cite{Imdad2007} obtained coincidence and fixed points of mappings in symmetric spaces under strict contractions. Further useful results in this direction were proved in \cite{Aamri2005, Aamri2003, Aranelovic2012, Cho2008, Shahzad2016}.

\begin{definition}[semi-metric space, \cite{Wilson1931}]
	\label{Def:definition-1}
	 Let $Y$ be a nonempty set. A function $\mathrm{d}_{s} : Y \times Y \rightarrow \mathbb{R_{+}}$ is known as (Hausdorff) semi-metric on $Y$ if for all$\ r, t \in Y$, the following conditions hold:
	\begin{enumerate}[label=\textup{(\roman*)}, ref=(\roman*)]
		\item $\mathrm{d}_{s}(r,t) \geq 0$ and equal to zero if and only if $r=t$;
		\item $\mathrm{d}_{s}(r,t) = \mathrm{d}_{s}(t,r)$.
	\end{enumerate}
	$(Y, \mathrm{d}_{s})$ is called a (Hausdorff) semi-metric space.
\end{definition}
Let $(Y, \mathrm{d}_{s})$ be a semi-metric space. Let $y \in Y$ and any $\varrho > 0$.
An open ball centered at $y$ with radius $\varrho$ is defined as follows,
$\mathrm{B}(y, \varrho) = \{x \in Y : \mathrm{d}_{s}(y, x) < \varrho \}.$
A topology $\tau_{\mathrm{d}_{s}}$ on $Y$ is represented by open sets as
\begin{align*}
	\tau_{\mathrm{d}_{s}} = \{U \subseteq Y : \forall \text{ } u \in U~ \exists \text{ } \varrho \in \mathbb{R_{+}} \text{ such that } B(u, \varrho) \subseteq U \}.
\end{align*}

\begin{definition}
	\label{Def:definition-2}
	Let $(Y, \mathrm{d}_{s})$ be a semi-metric space and a sequence $\{u_{n}\} \subseteq (Y, \mathrm{d}_{s})$. Then for every $u \in Y$ and $\varrho > 0$, the open ball $\mathrm{B}(u, \varrho)$ is a neighbourhood of $u$ with respect to the topology $\tau_{\mathrm{d}_{s}}$. Also, we say that $\lim\limits_{n \rightarrow \infty}\mathrm{d}_{s}(u_{n}, u) = 0$ if and only if $\{u_{n}\}$ converges to $u$ in $\tau_{\mathrm{d}_{s}}$.
	
The following properties are in place of the \emph{traingle inequality} for semi-metric spaces. Let $\{u_{n}\}, \{v_{n}\}$ and $\{w_{n}\}$ be sequence in a semi-metric space $(Y, \mathrm{d}_{s})$ and $u, v \in Y$. Then
	
\begin{enumerate}[label=\textup{\alph*)}, ref=\alph*]
		\item $\lim\limits_{n \rightarrow \infty}\mathrm{d}_{s}(u_{n}, u) = 0$ and $\lim\limits_{n \rightarrow \infty}\mathrm{d}_{s}(u_{n}, v) = 0 \Longrightarrow u = v$.
		\item $\lim\limits_{n \rightarrow \infty}\mathrm{d}_{s}(u_{n}, u) = 0$ and $\lim\limits_{n \rightarrow \infty}\mathrm{d}_{s}(u_{n}, v_{n}) = 0 \Longrightarrow \lim\limits_{n \rightarrow \infty}\mathrm{d}_{s}(v, u) = 0$.
		\item $\lim\limits_{n \rightarrow \infty}\mathrm{d}_{s}(u_{n}, v_{n}) = 0$ and $\lim\limits_{n \rightarrow \infty}\mathrm{d}_{s}(v_{n}, w_{n}) = 0 \Longrightarrow \lim\limits_{n \rightarrow \infty}d(u_{n}, w_{n}) = 0$.
		\item $\lim\limits_{n \rightarrow \infty}\mathrm{d}_{s}(u_{n}, v_{n}) = 0$ and $\lim\limits_{n \rightarrow \infty}\mathrm{d}_{s}(v_{n}, w_{n}) = 0 \Longrightarrow \lim\limits_{n \rightarrow \infty}\mathrm{d}_{s}(u_{n}, w_{n}) < \infty$.
		\item $\lim\limits_{n \rightarrow \infty}\mathrm{d}_{s}(u_{n}, u) = 0 \Longrightarrow \lim\limits_{n \rightarrow \infty}\mathrm{d}_{s}(u_{n}, v) = \mathrm{d}_{s}(u, v)$.
\end{enumerate}
\end{definition}
\noindent Below we provide the definition for a sequence $\{u_{n}\}$ to be Cauchy and complete in a semi-metric space.
\begin{definition}
	\label{Def:definition-3}
Let $\{u_n\}$ be a sequence in a semi-metric space $(Y, \mathrm{d}_{s})$.
\begin{enumerate}[label=\textup{\arabic*)}, ref=\arabic*]
\item If for any given $\epsilon > 0$, there exists $n(\epsilon) \in \mathbb{N}$ such that $\mathrm{d}_{s}(u_{n}, u_{m}) < \epsilon$ for all $n, m \geq n(\epsilon)$, then $\{u_{n}\}$ is called a $d$-Cauchy sequence.
\item $(Y, \mathrm{d}_{s})$ is called $S$-complete, if for each $\mathrm{d}_{s}$-Cauchy sequence $\{u_{n}\}$ in $Y$, there exists an element $u \in Y$ such that $\lim\limits_{n \rightarrow \infty} \mathrm{d}_{s}(u_{n}, u) = 0$.
\item $(Y, \mathrm{d}_{s})$ is called $\mathrm{d}_{s}$-Cauchy complete, if every $\mathrm{d}_{s}$-Cauchy sequence $\{u_{n}\}$ in $Y$ converges to some point $u \in Y$.
\end{enumerate}
\end{definition}

\begin{definition}
	\label{Def:definition-4-PH-metric}
	Let $(Y, \mathrm{d}_{s})$ be a semi-metric space. For $A, B \in CB(Y)$, define the map  			$\mathcal{H}_{s}:CB(Y) \times CB(Y) \rightarrow \mathbb{R_{+}}$ as
	\begin{align*}
		\mathcal{H}_{s}(A, B) = \mathrm{max}\{\mathcal{D}(A, B), \mathcal{D}(B, A)\},
	\end{align*}
where, $\mathcal{D}(A, B) = \mathrm{sup}\{\mathrm{d}_{s}(a, B): a \in A\}$ and $\mathrm{d}_{s}(x, A) = \mathrm{inf}\{\mathrm{d}_{s}(x, a) : a \in A\}$, and $CB(Y)$ is the closed and bounded subsets of the set $Y$.
The mapping $\mathcal{H}_{s}$ is called the Pompieu-Hausdorff $(PH)$-semi-metric induced by $\mathrm{d}_{s}$.
\end{definition}
\begin{definition}
	\label{Def:definition-5-closed-semi-metric}
	Let $(Y, \mathrm{d}_{s})$ be a semi-metric space and $\emptyset \neq U \subseteq Y$. $U$ is said to be \textbf{$\mathrm{d}_{s}$-closed}, if $\overline{U}=U$, where $\overline{U} = \{u \in Y : \mathrm{d}_{s}(u, U) = 0\}$.
\end{definition}
Clearly, if $(Y, \mathrm{d}_{s})$ is a $\mathrm{d}_{s}$-Cauchy complete semi-metric,
then $(CB(Y), \mathcal{H}_{s})$ is also a $\mathrm{d}_{s}$-Cauchy complete semi-metric space.
\begin{definition}
	\label{Def:definition-6-epsilon-chainable}
	A semi-metric space $(Y, \mathrm{d}_{s})$ is said to be $\epsilon$-chainable space for some given $\epsilon > 0$, there is $n \in \mathbb{N}$ and a sequence $\{y_{n}\}$ such that
	\begin{align*}
		y_{0} = u, \quad y_{n} = v \quad \text{ and } \quad \mathrm{d}_{s}(x_{i-1}, x_{i}) < \epsilon \text{ for } i = 1,  \ldots, n.
	\end{align*}
\end{definition}
Analogous results to that of Nadler's Lemmas as in \cite{Nadler1969} are now stated for the semi-metric space.
\begin{lemma}
	\label{Lemma:lemma-1-Nadler-analogue}
	Let $(Y, \mathrm{d}_{s})$ be a semi-metric space and $U, V \in CB(Y)$. Then $\mathrm{d}_{s}(u, V) \leq \mathcal{H}_{s}(U, V)$, whenever $u \in U$.
\end{lemma}
\begin{proof}
	If $u \in U$, then
	$$ \mathrm{d}_{s}(u , V) = \inf\limits_{v \in V} \mathrm{d}_{s}(u, v)  \leq \sup\limits_{u \in U} \mathrm{d}_{s}(u, V) \leq \mathcal{H}_{s}(U, V).
	$$
\end{proof}
\begin{lemma}
	\label{Lemma:lemma-2}
	Let $(Y, \mathrm{d}_{s})$ be a semi-metric space and $U, V \in CB(Y)$. Then for any $u \in U$, there exists $v \in V$ such that for any $\mu > 1$ implies $\mathrm{d}_{s}(u, v) \leq \mu \mathcal{H}_{s}(U, V)$.
\end{lemma}
\begin{proof}
	Suppose by way of contradiction that there exists $u \in U$ such that for every $v \in V$, we have $d(u, v) > \mu \mathcal{H}_{s}(U, V)$. By Lemma \ref{Lemma:lemma-1-Nadler-analogue}, it follows that
	\begin{align*}
		\mu \mathcal{H}_{s}(U, V) < \mathrm{d}_{s}(u, V) \leq \mathcal{H}_{s}(U, V)
	\end{align*}
which is a contradiction.
\end{proof}
\begin{lemma}[\cite{Jachymski1995}]
	\label{Lemma:lemma-3-Jachymski1995}
	If $U, V \in CB(Y)$ with $\mathcal{H}_{s}(U, V) < \epsilon$, then for each $u \in U$ there exists an element $v \in V$ such that $\mathrm{d}_{s}(u, v) < \epsilon$.
\end{lemma}
We consider the contraction $f: Y \rightarrow Y$ in a semi-metric space and state an associated result. Furthermore, we also identify $(X, d)$ with a directed graph $G$.

\begin{definition}
	\label{Def:defintion-7-contraction-boundedness}
	Let $(Y, \mathrm{d}_{s})$ be a semi-metric space. $f: Y \rightarrow Y$ is said to be a contraction, if there exists $\mathrm{c} \in [0, 1)$ such that
	\begin{align*}
		\mathrm{d}_{s}(fx, fy) \leq \mathrm{c} \mathrm{d}_{s}(x, y) \text{ for all } x, y \in Y.
	\end{align*}
\end{definition}

\noindent We say the semi-metric space $(Y, \mathrm{d}_{s})$ is bounded, if there is a real number $\mathfrak{M}$ such that $\mathfrak{M} < \infty$. Where, $\mathfrak{M} = \sup\{\mathrm{d}_{s}(x, y): x, y \in Y\}$.

\begin{lemma}[\cite{Jachymski1995}]
	\label{Lemma:lemma-4}
	Let $(Y, \mathrm{d}_{s})$ be a bounded $\mathrm{d}_{s}$-Cauchy semi-metric space and $f: Y \rightarrow Y$ a contraction. Then $f$ has a unique fixed point $y^{*} \in Y$. Furthermore, for each $y_{0} \in Y$, the iterative sequence $\{y_{0}, fy_{0}, f^{2}y_{0}, \ldots\}$ converges to $y^{*}$.
\end{lemma}

\section{Graphic Contractive Mappings in Semi Metric Spaces}
Jachymski in \cite{Jachymski2008} describes his construction of imposing a graph structure on a metric space. Ran and Reurings' result in \cite{RanReurings2004}, as mentioned previously was the main motivation for such an effort towards generalization of the order or poset structure of a metric space. The usefulness of imposing a graph structure on $(Y, \mathrm{d}_{s})$ is that it generalizes the ordered structure on $(Y, \mathrm{d}_{s})$. One the first result in fixed-point theory with a poset structure was obtained by Ran and Reurings in \cite{RanReurings2004}. The latter result was the main thrust in the generalization of equipping spaces with a graph structure which was motivated by Jachymski in \cite{Jachymski2008}. Recently, Latif et al. \cite{Latif2019} obtained some fixed point results for class of set-contraction mappings endowed with a directed graph. In this chapter, we established several fixed point results with respect to generalized
rational contractive mappings in semi-metric spaces endowed with a directed graph. These results extend and generalized several comparable results in \cite{Frechet1906, Hicks1998, Latif2012, MotPetrusel2009}.

Below we define on how to identify a metric space with a graph $G$.
\begin{definition}
\label{Def:definition-8-graph-id}
Let $Y$ be a non-empty set. We identify $Y$ with a directed graph $G = (V(G), E(G))$, as follows:
\begin{enumerate}[label=\textup{\arabic*)}, ref=\arabic*]
\item $Y = V(G)$, \text{where $V(G)$ is the vertex set};
\item $\Delta \subseteq E(G)$, where $\Delta \subseteq Y \times Y$ is the diagonal \text{and $E(G)$ is the edge set};
\item Loops at a vertex are allowed, for all $v \in V(G)$;
\item Parallel edges between distinct vertices are prohibited.
\end{enumerate}
Since $V(G) = Y$, we shall freely refer to elements of $V(G)$ or $Y$ interchangeably.
\end{definition}
The semi-metric space $(Y, \mathrm{d}_{s})$ is identified with a directed graph $G$ as in Definition \ref{Def:definition-8-graph-id} such that $G$ is weighted, that is, for each $(u, v) \in E(G)$ we assign the weight $\mathrm{d}_{s}(u, v)$. Where loops have weight $\mathrm{d}_{s}(u, u) = 0$ for all $u \in V(G)$. Moreover, the $PH$-semi-metric induced by the semi-metric $\mathrm{d}_{s}$ assigns a non-zero $PH$-weight to every $U, V \in CB(Y)$ and $\mathcal{H}_{s}(U, V) = 0$, whenever $U=V$.
If $u,v\in V(G)$, then a (directed) path in $G$ between $u$ and $v$
having length $l\in
\mathbb{N}$ is a finite sequence $\{y_{n}\}$ $(n\in \{0,\dots,l\})$ of
the vertices so that $u=y_{0}$, $v=y_{l}$ also $
(y_{i-1}, y_{i})\in E(G)$ for $i \in \{1, \ldots,l\}.$
If between every two vertices there is a (directed) path, then $G$
is called connected. If $\widetilde{G}$ is connected, then $G$ is said to
be weakly connected, where $\widetilde{G}$ represents the undirected graph
produced by removing the direction of edges of $G$. $G^{-1}$ is the graph formed by reversing the edge direction in $G$, that is $(v, u) \in E(G^{-1})$, only if $(u, v) \in E(G)$.
We also use notation $Y_{f}=\{y\in Y:(y,fy)\in E(G)\}.$

\begin{definition}[\cite{Jachymski2008}]
	\label{Def:definition-9-banach-g-contraction}
	A map $f:Y \rightarrow Y$ is known as Banach $G$-contraction when
\begin{enumerate}[label=\textup{\alph*)}, ref=\alph*]
		\item $f$ conserves edges of graph $G$, that is, if $(u, v) \in E(G)$, then $(fu, fv) \in E(G)$ for all $u, v \in V(G) = Y$.
		\item $f$ shrinks the weights of edges of graph $G$, that is, there exits $k \in (0,1)$ such that for every $u, v \in V(G)$ and $(u,v)\in E(G)$ implies $\gamma(fu,fv)\leq k \gamma(x,y)$.
	\end{enumerate}
\end{definition}
Jachymski in \cite{Jachymski2008} utilized the property $\mathcal{P}$ given as follows: For any sequence $\{y_{n}\}$\ in $Y$, if $\lim\limits_{n \rightarrow \infty}(y_n) = y$ and also
$(y_{n}, y_{n+1})\in E(G)$, then $(y_{n}, y)\in E(G).$
\begin{theorem}[\cite{Jachymski2008}]
	Suppose $(Y, \gamma)$ is a complete metric space and a directed graph $G$ is having property $\mathcal{P}$ with $G$-contraction $f:Y \rightarrow Y$. Then,
	
	\begin{enumerate}[label=\textup{\arabic*)}, ref=\arabic*]
		\item $Y_{f}\neq \emptyset $ if and only if $f$ admits a fixed point.
		\item $f$ is a Picard operator, only if $G$ is weakly connected and $Y_{f}\neq \emptyset$.
		\item $f\mid_{[y]_{\widetilde{G}}}$ is a Picard operator for every $y \in Y_{f}$.
		\item If $f$ is a weakly Picard operator, then $Y_{f} \times Y_{f}\subseteq E(G)$.
	\end{enumerate}
\end{theorem}

\noindent Deducing the domain, expanding the range, or expanding the contractive condition of the mappings have all been used to extend the Banach contraction principle. Nadler in \cite{Nadler1969} used $PH$-metric to show the multi-valued type of the Banach contraction principle by supplanting the mapping's range with $CB(X)$, for some arbitrary set $X$.

\begin{definition}[\cite{Nadler1969}]
	\label{Def:definition-10-Nadler-set-valued-contraction}
	Consider the two metric spaces $(X,d_{1})$ and $(Y,d_{2}).$ And $f:X \rightarrow CB(Y)$ is a
	set-valued Lipschitz mapping of $X$ into $Y$ if and only if for any $x,z\in
	X$,
	\begin{equation*}
		H(fx,fz)\leq \alpha d_{1}(x,z),
	\end{equation*}
where $\alpha \geq 0$ is a fixed real number and $H$ is the $PH$-metric. If $\alpha < 1$. The map $f$ is
then referred to as a set-valued contraction.
\end{definition}

\begin{theorem}[\cite{Nadler1969}]
	Assume $(X,d)$ is a complete metric space. If $f:X\rightarrow CB(X)$ is a set-valued contraction,
	then $f$ possess a fixed point.
\end{theorem}

\begin{definition}
	\label{Def:definition-11-path-edge-in-graph}
	Letting $\emptyset \neq V, W \subseteq Y$. Then,
	\begin{enumerate}[label=\textup{\arabic*)}, ref=\arabic*]
		\item when we declare that there exists an edge among $V$ and $W$, this conveys that there exists an edge among $v\in V$ and $w\in W$ that may express as $(V,W)\subseteq E(G).$
		\item when we declare that there is a path between $V$ and $W$, we are referring to a path between $v\in V$ and $w\in W$.
	\end{enumerate}
\end{definition}
We define a relation $\mathcal{R}$ on $CB(Y)$ as follows: $U\mathcal{R}V$ for $U, V\in CB(Y)$ if and only if there is a path between $U$ and $V$. A relation $\mathcal{R}$ on $CB(Y)$ is called transitive, if there is a path between $U$ and $V$ and there is a path between $V$ and $W$, then there is a path between $U$ and $W$. Furthermore, a subset $\mathcal{C} \subseteq CB(Y)$ is said to be complete if for any $U, V \in \mathcal{C}$, there is an edge between $U$ and $V$.
The following is the equivalence class of $U$ for $U\in CB(Y)$
activated by $\mathcal{R}$:
\begin{equation*}
	[U]_{G}=\{V\in CB(Y): U\mathcal{R}V\}.
\end{equation*}
Rather than the mapping $Y \rightarrow Y$ or $Y \rightarrow CB(Y),$ we
consider the mapping $T:CB(Y)\rightarrow CB(Y)$.

The set $Y_{T}$ for a map $T:CB(Y) \rightarrow CB(Y)$ is interpreted as
\begin{equation*}
	Y_{T}=\{U\in CB(Y):(U,T(U))\subseteq E(G)\}.
\end{equation*}

\begin{definition}[\cite{Abbas2015}]
	\label{Def:definition-12-propert-P-star}
We say a graph $G$ has the property $\mathcal{\overset{*}{P}}$
if, for a sequence $\{U_{n}\}$ in $CB(Y)$, $\lim\limits_{n \rightarrow \infty} U_{n} = U$ and
$(U_{n} , U_{n+1})\in E(G)$ for
$n\in \mathbb{N}$, implies that there exists a subsequence $\{U_{n_k}\}$ of
$\{U_{n}\}$ with $(U_{n_{k}}, U)\in E(G)$ for $n\in \mathbb{N}.$
\end{definition}

We present the following graph contractions that will be central to our results, namely,  \emph{generalized rational graph contraction}. These contractions are studied in detail by Abbas and Nazir in \cite{Abbas2015}.

\begin{definition}[\cite{Abbas2015}]
	\label{Def:definition-13-gen-graph-contraction}
	Let $T: CB(Y) \rightarrow CB(Y)$ be a set-valued mapping. Then $T$ is said to be a generalized rational graph $\lambda$-contraction, if the following conditions are satisfied:
	\begin{enumerate}[label=\textup{\arabic*)}, ref=\arabic*]
		\item An edge (path) between $U$ and $V$ implies an edge (path) between $T(U)$ and $T(V)$, for all $U, V \in CB(Y)$.
		\item There exist $\lambda \in [0, 1)$, such that:
		\begin{align*}
			\mathcal{H}_{s}\Big(T(U), T(V)\Big) \leq \lambda M_{T}(U, V) \text{ for all } U, V \in CB(Y),
		\end{align*}
	\end{enumerate}
	\noindent where
	\begin{align}
		M_{T}(U, V) &= \max\Big\{\mathcal{H}_{s}(U, V), \mathcal{H}_{s}(U, T(U)), \mathcal{H}_{s}(V, T(V)), \notag \\
		&\qquad \qquad \frac{\mathcal{H}_{s}(V, T(V))[1 + \mathcal{H}_{s}(U, T(U))]}{1 + \mathcal{H}_{s}(U, V)}, \notag \\
		&\qquad \qquad \frac{\mathcal{H}_{s}(V, T(U))[1 + \mathcal{H}_{s}(U, T(U))]}{1 + \mathcal{H}_{s}(U, V)} \Big\}. \label{semi_eqn01}
	\end{align}
\end{definition}

\section{Main Fixed Point Result on Semi Metric Space}
\begin{theorem}
	\label{main_theorem_3}
	Let $(Y, \mathrm{d}_{s})$ be a $\mathrm{d}_{s}$-Cauchy complete semi-metric space endowed with a directed graph, such that $V(G) = Y$ and $\Delta \subseteq E(G)$. If $T: CB(Y) \rightarrow CB(Y)$ is a generalized rational $\lambda$-contraction mapping and $F(T)=\{U\in CB(Y): T(U)=U \}$, then the following holds:
	\begin{enumerate}[label=\textup{\arabic*)}, ref=\arabic*]
		\item If $F(T) \neq \emptyset$ is complete, then the $PH$-weight assigned to $U, V \in F(T)$ is zero. \label{main01}
		\item If $F(T) \neq \emptyset$, then $Y_{T} \neq \emptyset$. \label{main02}
		\item If $Y_{T} \neq \emptyset$ and the weakly connected graph $\widetilde{G}$ satisfies $\mathcal{\overset{*}{P}}$, then $T$ has a fixed point. \label{main03}
		\item $F(T) \neq \emptyset$ is complete if and only if $F(T)$ is a singleton. \label{main04}
	\end{enumerate}
\end{theorem}

\begin{proof} \smartqed \ref{main01}) Suppose otherwise, that there exists $U, V \in F(T)$ such that $\mathcal{H}_{s}(U, V) > 0$. Now, since $T$ is a generalized $\lambda$-contraction, it follows that
	\begin{equation}
		\mathcal{H}_{s}(U, V) = \mathcal{H}_{s}(T(U), T(V)) \leq \lambda M_{T}(U, V)  \label{semi_eqn02}
	\end{equation}
for some $\lambda \in (0, 1)$, where
	\begin{align*}
		M_{T}(U, V) &= \max\Big\{\mathcal{H}_{s}(U, V), \mathcal{H}_{s}(U, T(U)), \mathcal{H}_{s}(V, T(V)), \notag \\
		&\qquad \qquad \frac{\mathcal{H}_{s}(V, T(V))[1 + \mathcal{H}_{s}(U, T(U))]}{1 + \mathcal{H}_{s}(U, V)}, \notag \\
		&\qquad \qquad \frac{\mathcal{H}_{s}(V, T(U))[1 + \mathcal{H}_{s}(U, T(U))]}{1 + \mathcal{H}_{s}(U, V)} \Big\} \\
		&= \mathcal{H}_{s}(U, V).
	\end{align*}
The above inequality (\ref{semi_eqn02}) implies
	\begin{align*}
		\mathcal{H}_{s}(U, V) \leq \lambda \mathcal{H}_{s}(U, V).
	\end{align*}
	This is a contraction. Hence our result follows.
		
	\noindent \ref{main02}) Let $U \in F(T)$, then since $(u , u) \subseteq E(G)$ for all $u \in V(G)$. That is  $(U, U) = (U, T(U)) \subseteq E(G)$. Therefore, we have  $U \in CB(Y)$ such that $(U, T(U)) \subseteq E(G)$. Hence, $U \in Y_{T}$.
	
	\noindent \ref{main03}) Let $U \in Y_{T} \neq \emptyset$. Then since $U \in CB(Y)$, and $\widetilde{G}$ is weakly connected, it follows that $CB(Y) \subseteq [U]_{\widetilde{G}} = \mathrm{P}(Y)$, a nonempty power set on $Y$. Since $T$ is a self-map and the equivalence class satisfies the transitive property on $CB(Y)$, we have $T(U) \in [U]_{\widetilde{G}}$. As such by a similar argument, we have $T(U_{i}) \in [U]_{\widetilde{G}}$ for each $U_{i} \in [U]_{\widetilde{G}}$. By definition of $T$, we have $\big(T^{n}(U), T^{n+1}(U)\big) \in E(\widetilde{G})$ for all $n \in \mathbb{N}$.
We now define a recursive iterative sequence, as follows:
	\begin{align*}
		U & = U_{0}, \\
		T(U_{0}) & = U_{1}, \\
		T^{2}(U_{0}) & = T(U_{1}) = U_{2}, \\
		& \cdots \\
		T^{n}(U_{0}) & = T(U_{n-1}) = U_{n}.
	\end{align*}
We assume that  $U_{n+1} \neq U_n$ for all $ n \in \{0,1,2,\dots\}$. In case  $U_{k+1} =  U_k$ for some $k$, then $ T(U_k) =U_{k+1}= U_{k}$, that is, $U_k$ is the fixed point of $T$. 	
	Since $\widetilde{G}$ is weakly connected, then there exists a sequence $\{y_{i}\}$ such that $(y_{i}, y_{i+1}) \in E(\widetilde{G})$, for $i = 0, \dots, n-1$ such that $y_{i} \in U_{i}$ for $i = 0,  \dots, n$. Again by definition of $T$ we have
\begin{equation}
\begin{array}{ll}
		\mathcal{H}_{s}\Big(T^{n}(U), T^{n+1}(U) \Big)  = \mathcal{H}_{s}\big(U_{n}, U_{n+1}\big) & = \mathcal{H}_{s}\big(T(U_{n-1}), T(U_{n})\big)  \\
& \leq \lambda \Big(M_{T}(U_{n-1}, U_{n})\Big)
	\end{array}
\label{semi_eqn03}
\end{equation}
for some $\lambda \in [0, 1)$, where,
\begin{align*}
M_{T}(U_{n-1}, U_{n}) &= \max\Big\{H(U_{n-1}, U_{n}), H(U_{n-1}, T(U_{n-1})), H(U_{n}, T(U_{n})), \\
		&\qquad \qquad \frac{H(U_{n}, T(U_{n})[1 + H(U_{n-1},  T(U_{n-1}))])}{1 + H(U_{n-1}, U_{n})}, \\ &\qquad \qquad \frac{H(U_{n}, T(U_{n-1}))[1 + H(U_{n-1},  T(U_{n-1}))])}{1 + H(U_{n-1}, U_{n})}\Big\}.
\end{align*}
Applying the mapping, we get
	\begin{align*}
		M_{T}(U_{n-1}, U_{n}) &= \max\Big\{\mathcal{H}_{s}(U_{n-1}, U_{n}), \mathcal{H}_{s}(U_{n-1}, U_{n}), \mathcal{H}_{s}(U_{n}, U_{n+1}), \\
		&\qquad \qquad \frac{\mathcal{H}_{s}(U_{n}, U_{n+1})[1 + \mathcal{H}_{s}(U_{n-1},  U_{n})])}{1 + \mathcal{H}_{s}(U_{n-1}, U_{n})}, \\ &\qquad \qquad \frac{\mathcal{H}_{s}(U_{n}, U_{n})[1 + \mathcal{H}_{s}(U_{n-1},  U_{n})])}{1 + \mathcal{H}_{s}(U_{n-1}, U_{n})}\Big\} \\
		&= \max\{\mathcal{H}_{s}(U_{n-1}, U_{n}), \mathcal{H}_{s}(U_{n}, U_{n+1})\}.
	\end{align*}
	\noindent Now, if $M_{T}(U, V) = \mathcal{H}_{s}(U_{n}, U_{n+1})$, then by inequality (\ref{semi_eqn03}), we have a contradiction. Therefore the only value $M_{T}(U_{n-1}, U_{n})$ can obtain is $\mathcal{H}_{s}(U_{n-1}, U_{n})$. It now follows that
	\begin{align*}
		\mathcal{H}_{s}\Big(T^{n}(U), T^{n+1}(U) \Big) & = \mathcal{H}_{s}\big(U_{n}, U_{n+1}\big)  = \mathcal{H}_{s}\big(T(U_{n-1}), T(U_{n})\big) \\
		& \leq \lambda \mathcal{H}_{s}\big(U_{n-1}, U_{n} \big) \\
		& = \lambda \mathcal{H}_{s}\big(T(U_{n-2}), T(U_{n-1}) \big) \\
		& \leq \lambda^{2}\mathcal{H}_{s}(U_{n-2}, U_{n-1}) \\
		\cdots & \leq \lambda^{n}\mathcal{H}_{s}(U_{0}, U_{1}) = \lambda^{n}\mathcal{H}_{s}(U, T(U)).
	\end{align*}
	\noindent Now, for $m, n \in \mathbb{N}$,
	\begin{align*}
		\mathcal{H}_{s}\Big(T^{n}(U), T^{m + n}(U) \Big) & \leq \lambda^{n} \mathcal{H}_{s}(U_{0}, U_{m}) = \lambda^{n} \mathcal{H}_{s}(U, U_{m}).
	\end{align*}
Taking the limit $n \rightarrow \infty$ yields $\lambda^{n} \mathcal{H}_{s}(U, U_{m}) \rightarrow 0$. Therefore, $\{T^{n}(U)\}$ is a $\mathrm{d}_{s}$-Cauchy sequence in $CB(Y)$. And since $(Y, d)$ is $\mathrm{d}_{s}$-Cauchy complete, there exists $U^{*}$ such that $\lim\limits_{n \rightarrow \infty} T^{n}(U) = U^{*}$.
	
	\noindent Now, we show that $U^{*}$ is a fixed point of $T$. Since, $\{T^{n}\}$ with $\lim\limits_{n \rightarrow \infty} T^{n}(U) = U^{*}$ is a $\mathrm{d}_{s}$-Cauchy sequence and by our assumption we have  $(T^{n}(U), T^{n+1}(U)) \in E(\tilde{G})$ for every $n \in \mathbb{N}$. Therefore, by property $\mathcal{\overset{*}{P}}$, we have a subsequence $\{T^{n_k}(U)\}$ of $\{T^{n}(U)\}$ such there is  an edge between $T^{n_k}(U)$ and $U^{*}$ for all $n \in \mathbb{N}$. It follows that
	\begin{align*}
		\mathcal{H}_{s}\Big(T(U^{*}), T^{n_k}(U^{*}) \Big) & \leq \lambda M_{T}(U^{*}, T^{n_k - 1}(U^{*}))
	\end{align*}
for some $\lambda \in [0, 1)$. Note that
$M_{T}(U^{*}, T^{n_k - 1}(U^{*})) = M_{T}(T^{n_k - 1}(U^{*}) ,U^{*})$, where
	\begin{align*}
		& M_{T}(T^{n_k - 1}(U^{*}), U^{*}) \\
& = \max\Big\{\mathcal{H}_{s}(T^{n_k - 1}(U^{*}), U^{*}), \mathcal{H}_{s}(T^{n_k - 1}(U^{*}), T(T^{n_k - 1}(U^{*}))), \mathcal{H}_{s}(U^{*}, T(U^{*})), \\
		&\qquad \qquad \frac{\mathcal{H}_{s}(U^{*}, T(U^{*}))[1 + \mathcal{H}_{s}(T^{n_k - 1}(U^{*}), T(T^{n_k - 1}(U^{*})))]}{1 + \mathcal{H}_{s}(T^{n_k - 1}(U^{*}), U^{*})},
		\\&\qquad \qquad \frac{\mathcal{H}_{s}(U^{*}, T(T^{n_k - 1}(U^{*})))[1 + \mathcal{H}_{s}(T^{n_k - 1}(U^{*}), T(T^{n_k - 1}(U^{*})))]}{1 + \mathcal{H}_{s}(T^{n_k - 1}(U^{*}), U^{*})} \Big\}.
	\end{align*}
Applying the mapping $T$, we get
	\begin{align*}
&		M_{T}(U^{*}, T^{n_k - 1}(U^{*})) \\
&= \max\Big\{\mathcal{H}_{s}(T^{n_k - 1}(U^{*}), U^{*}),  \mathcal{H}_{s}(T^{n_k - 1}(U^{*}), T^{n_k}(U^{*})),  \mathcal{H}_{s}(U^{*}, T(U^{*})), \\
		&\qquad \qquad  \frac{\mathcal{H}_{s}(U^{*}, T(U^{*}))[1 + \mathcal{H}_{s}(T^{n_k - 1}(U^{*}), T^{n_k}(U^{*}))]}{1 + \mathcal{H}_{s}(T^{n_k - 1}(U^{*}), U^{*})},
		\\ &\qquad \qquad \frac{\mathcal{H}_{s}(U^{*}, T^{n_k}(U^{*}))[1 + \mathcal{H}_{s}(T^{n_k - 1}(U^{*}), T(T^{n_k - 1}(U^{*})))]}{1 + \mathcal{H}_{s}(T^{n_k - 1}(U^{*}), U^{*})} \Big\}.
	\end{align*}
We now consider the following cases.
	
	\noindent \textbf{Case 1}: If $M_{T}(U^{*}, T^{n_k - 1}(U^{*})) = \mathcal{H}_{s}(T^{n_k - 1}(U^{*}), U^{*})$, taking the limit for $n \rightarrow \infty$ yields
	\begin{align*}
		\lim\limits_{n \rightarrow \infty} \mathcal{H}_{s}\Big(T^{n_k}(U^{*}), T(U^{*}) \Big) &\leq \lambda \lim\limits_{n \rightarrow \infty} \mathcal{H}_{s}(T^{n_k - 1}(U^{*}), U^{*}) \\
		&=\lambda \mathcal{H}_{s}(U^{*}, U^{*}) = 0.
	\end{align*}
	Hence, we must have $\lim\limits_{n \rightarrow \infty}\mathcal{H}_{s}\Big(T^{n_k}(U^{*}), T(U^{*}) \Big) = 0$, that is $T(U^{*}) = U^{*}$. \newline
	
	\noindent \textbf{Case 2}: If $M_{T}(U^{*}, T^{n_k - 1}(U^{*})) = \mathcal{H}_{s}(T^{n_k - 1}(U^{*}), T^{n_k}(U^{*}))$, taking the limit for $n \rightarrow \infty$ yields
	\begin{align*}
		\lim\limits_{n \rightarrow \infty} \mathcal{H}_{s}\Big(T(U^{*}), T^{n_k}(U^{*}) \Big) &= \mathcal{H}_{s}(U^{*}, T(U^{*})) \\
		& \leq \lambda \lim\limits_{n \rightarrow \infty} \mathcal{H}_{s}(T^{n_k - 1}(U^{*}), T^{n_k}(U^{*})) \\
		& = \lambda \mathcal{H}_{s}(U^{*}, U^{*}) = 0.
	\end{align*}
We must have $\mathcal{H}_{s}(U^{*}, T(U^{*})) = 0$, hence $T(U^{*}) = U^{*}$. \newline
	
	\noindent \textbf{Case 3}: If $M_{T}(U^{*}, T^{n_k - 1}(U^{*})) =  \mathcal{H}_{s}(U^{*}, T(U^{*}))$ taking the limit for $n \rightarrow \infty$ yields
	\begin{align*}
		\lim\limits_{n \rightarrow \infty} \mathcal{H}_{s}\Big(T(U^{*}), T^{n_k}(U^{*}) \Big) &= \mathcal{H}_{s}(U^{*}, T(U^{*})) \leq \lambda \lim\limits_{n \rightarrow \infty} \mathcal{H}_{s}(U^{*}, T(U^{*})) \\
		& \leq \lambda \mathcal{H}_{s}(U^{*}, T(U^{*})).
	\end{align*}
	Therefore $\mathcal{H}_{s}(U^{*}, T(U^{*})) = 0$, and we conclude:  $T(U^{*}) = U^{*}$. \newline
	
	\noindent \textbf{Case 4}: If $M_{T}(U^{*}, T^{n_k - 1}(U^{*})) = \frac{\mathcal{H}_{s}(U^{*}, T(U^{*}))[1 + \mathcal{H}_{s}(T^{n_k - 1}(U^{*}), T^{n_k}(U^{*}))]}{1 + \mathcal{H}_{s}(T^{n_k - 1}(U^{*}), U^{*})}$, taking the limit $n \rightarrow \infty$ yields
	\begin{align*}
		\lim\limits_{n \rightarrow \infty} \mathcal{H}_{s}\Big(T(U^{*}), T^{n_k}(U^{*}) \Big)  &\leq \lambda \lim\limits_{n \rightarrow \infty}\frac{\mathcal{H}_{s}(U^{*}, T(U^{*}))[1 + \mathcal{H}_{s}(T^{n_k - 1}(U^{*}), T^{n_k}(U^{*}))]}{1 + \mathcal{H}_{s}(T^{n_k - 1}(U^{*}), U^{*})} \\
		&=\lambda \mathcal{H}_{s}(U^{*}, T(U^{*})).
	\end{align*}
	\noindent As in \emph{Case 3}, we have   $T(U^{*}) = U^{*}$. \newline
	
	\noindent \textbf{Case 5}: If $M_{T}(U^{*}, T^{n_k - 1}(U^{*})) = \frac{\mathcal{H}_{s}(U^{*}, T^{n_k}(U^{*}))[1 + \mathcal{H}_{s}(T^{n_k - 1}(U^{*}), T(T^{n_k - 1}(U^{*})))]}{1 + \mathcal{H}_{s}(T^{n_k - 1}(U^{*}), U^{*})}$, taking the limit for $n \rightarrow \infty$ yields
\begin{align*}
		& \lim\limits_{n \rightarrow \infty} \mathcal{H}_{s}\Big(T(U^{*}), T^{n_k}(U^{*}) \Big) \\
&\leq \lambda \lim\limits_{n \rightarrow \infty} \frac{\mathcal{H}_{s}(U^{*}, T^{n_k}(U^{*}))[1 + \mathcal{H}_{s}(T^{n_k - 1}(U^{*}), T(T^{n_k - 1}(U^{*})))]}{1 + \mathcal{H}_{s}(T^{n_k - 1}(U^{*}), U^{*})} \\
		&=\lambda(0) = 0.
	\end{align*}
As in \emph{Case 3}, we conclude   $T(U^{*}) = U^{*}$. Therefore, $U^{*} \in F(T)$.
We now show that the fixed point $U^{*}$ is unique. Suppose by way of contradiction that we have $U^{*}, V^{*} \in F(T)$ and let $\mathcal{H}_{s}(U^{*}, V^{*}) > 0$. Then by a similar argument to (\ref{main01}) of Theorem \ref{main_theorem_3}, we have  $\mathcal{H}_{s}(U^{*}, V^{*}) \leq \lambda \mathcal{H}_{s}(U^{*}, V^{*})$, whence $\mathcal{H}_{s}(U^{*}, V^{*}) = 0$ implying $U^{*} = V^{*}$. We also conclude that $|F(T)| = 1$.

	\noindent \ref{main04}) Letting $F(T) \neq \emptyset$ be complete, then by (\ref{main01}) of Theorem \ref{main_theorem_3}, we conclude that $F(T)$ is a singleton. This proves the necessity. Proving the sufficiency, let $F(T)$ be a singleton with $U \in F(T)$, then by (\ref{main02}) of Theorem \ref{main_theorem_3}, we have  $U \in Y_{T}$, hence $F(T)$ is complete.
\qed \end{proof}

\noindent To illustrate the above result, we provide an Example with respect to a semi-metric space $(Y, \mathrm{d}_{s})$.
\begin{example}
Let $Y = \{0, 1, 4\} = V(G)$ and equipped with semi-metric $\mathrm{d}_{s}: Y \times Y \rightarrow \mathbb{R^{+}}$ where
	$
		\mathrm{d}_{s}(x - y) = (x - y)^{2}.
	$
	Clearly the semi-metric $\mathrm{d}_{s}$ does not satisfy the triangle inequality.
Let the edge set be defined as
	\begin{align*}
		E(G) = \{(0, 0), (1, 1), (4, 4), (0, 1), (0, 4), (1, 4)\}.
	\end{align*}
Note that the sets $\{0\}, \{0, 1\} \text{ and } \{0, 4\}$ are all bounded and closed. We demonstrate that this is so for the set $\{0, 4\}$, the others follow a similar argument.
\begin{align*}
y \in \overline{\{0, 4\}} &\Leftrightarrow d(y, \{0, 4\}) = 0 \Leftrightarrow \min\{d(y, 0), d(y, 4)\} = 0 \Leftrightarrow y \in \{0, 4\}.
\end{align*}
Hence, $\{0, 4\}$ is closed with respect to the semi-metric $\mathrm{d}_{s}$. Boundedness is clear.

The $K_{3}$ graph  with $V(G)$ and $E(G)$ above is shown in Figure \ref{fig:k3_graph1}, with the Pompeiu-Hausdorff weight, as defined below.
\begin{figure}[h]
		\centering
		\begin{tikzpicture}
			\tikzstyle{vertex} = [circle, fill=black!10]
			\tikzstyle{edge} = [->, very thick]			
			\node[vertex](v1) at (0, 0) {$\bar{0}$};
			\node[vertex](v2) at (5, 0) {$\bar{1}$};
			\node[vertex](v3) at (2.5, -3) {$\bar{4}$};			
			\draw[edge] (v1) to [out=200,in=90,looseness=9] (v1);
			\draw[edge] (v2) to [out=350,in=90,looseness=9] (v2);
			\draw[edge] (v3) to [out=350,in=270,looseness=9] (v3);
			\draw[edge] (v1)--(v2);
			\draw[edge] (v1)--(v3);
			\draw[edge] (v2)--(v3);		
			\draw[-] (v1) edge  node[pos=0.5, above] {$\mathcal{H}_{s}(\bar{0}, \bar{1}) = 1$} (v2);
			\draw[-] (v1) edge  node[pos=0.5, left] {$16 = \mathcal{H}_{s}(\bar{0}, \bar{4})$} (v3);
			\draw[-] (v2) edge  node[pos=0.5, right] {$\mathcal{H}_{s}(\bar{1}, \bar{4}) = 9$} (v3);
		\end{tikzpicture}
		\caption{$K_{3}$ - Pompeiu-Hausdorff Weighted Graph with Loops} \label{fig:k3_graph1}
	\end{figure}
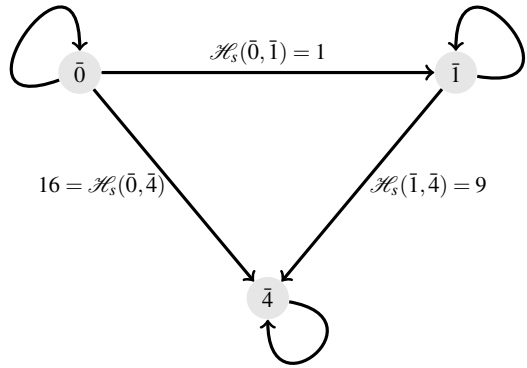

We use the following notation $\bar{0} = \{0\}, \bar{1} = \{0, 1\}$ and $\bar{4} = \{0, 4\}$, where, $CB(X) = \{\bar{0}, \bar{1}, \bar{4}\}$. We now employ the $PH$-metric to calculate the distances between the elements of $CB(Y)$.
	\begin{align*}
		\mathcal{H}_{s}(U, V) = \begin{cases}
			0 \text{ if } U = V = \bar{0}, U = V = \bar{1} \text{ or } U = V = \bar{4} \\
			1 \text{ if } U = \bar{0} \text{ and } V = \bar{1} \\
			9 \text{ if } U = \bar{1} \text{ and } V = \bar{4} \\
			16 \text{ if } U = \bar{0} \text{ and } V = \bar{4}.
		\end{cases}
	\end{align*}
We now define the mapping $T: CB(Y) \rightarrow CB(Y)$, as follows
	\begin{align*}
		T(V) = \begin{cases}
			\bar{0} \text{ if } V \in \{\bar{0}, \bar{1} \}\\
			\bar{1} \text{ if } V = \bar{4}.
		\end{cases}
	\end{align*}
Notice, that between any two elements $U, V \in CB(Y)$ there is  an edge (path) between them, and an edge (path) between $T(U)$ and $T(V)$.

We now consider the following cases for the elements of $CB(Y)$.
	
	\noindent \textbf{Case 1}: If $U, V \in \{\bar{0},  \bar{1} \}$ or $U = V = \bar{4}$, then $\mathcal{H}_{s}(T(U), T(V)) = 0$.\newline
	
	\noindent \textbf{Case 2}: If $U \in \{\bar{0}, \bar{1} \}$ and $V = \bar{4}$, then we have two sub-cases, as follows:
\begin{enumerate}[label=\textup{(\roman*)}, ref=\textup{(\roman*)}]
\item Whenever $\lambda \in (\frac{1}{16}, 1)$,
	\begin{align*}
		\mathcal{H}_{s}(T(\bar{0}), T(\bar{4})) = \mathcal{H}_{s}(\bar{0}, \bar{1}) = 1 < \lambda \mathcal{H}_{s}(\bar{0}, \bar{4}) = 16\lambda;
	\end{align*}
\item Whenever $\lambda \in (\frac{1}{9}, 1)$,
	\begin{align*}
		\mathcal{H}_{s}(T(\bar{1}, \bar{4})) &= \mathcal{H}_{s}(\bar{0}, \bar{1}) = 1 < \lambda \mathcal{H}_{s}(\bar{1}, \bar{4}) = 9\lambda.
	\end{align*}
\end{enumerate}
Thus all requirements of Theorem \ref{main_theorem_3} are satisfied, with $T(\bar{0}) = \bar{0}$ and $F(T)$ is complete.
\end{example}

\begin{corollary}
	Let $(Y, \mathrm{d}_{s})$ be a $\mathrm{d}_{s}$-Cauchy complete semi-metric space equipped with a directed graph $G$ such that $V(G) = Y$ and $\Delta \subseteq E(G)$. If $G$ is a weakly connected, then the generalized rational graph $\lambda$-contraction mapping $T : CB(Y) \rightarrow CB(Y)$ with $(U, V) \subseteq E(G)$ where $V \in T(U)$, attains a fixed point.
\end{corollary}

\begin{corollary}
	Let $(Y, \mathrm{d}_{s})$ be an $\epsilon$-chainable space for some $\epsilon > 0$, the mapping $T : CB(Y) \rightarrow CB(Y)$ and $\lambda \in [0, 1)$ with $\mathcal{H}_{s}(U, V) < \epsilon$. Then $T$ has a fixed point, whenever, $\mathcal{H}_{s}(T(U), T(V)) \leq \lambda M_{T}(\mathcal{H}_{s}(U, V))$ for all $U, V \in CB(Y)$.
\end{corollary}

\begin{proof} \smartqed From our assumption and Lemma \ref{Lemma:lemma-3-Jachymski1995}, for every $u \in U$ there exists $v \in V$ such that $d(u, v) < \epsilon$. Now, since $(Y, \mathrm{d}_{s})$ is an $\epsilon$-chainable space we have by Definition \ref{Def:definition-6-epsilon-chainable}, a sequence $\{y_{n}\}$ where $d(y_{i}, y_{i+1}) < \epsilon$ for $i = 0, \ldots, n-1$.
Equipping the semi-metric space with a graph, as follows: $V(G) = Y$ and $E(G) = \{(y_{i}, y_{i+1}) : d(y_{i}, y_{i+1}) < \epsilon \}$, it is easy to see that $G$ is weakly connected, hence $(Y_{i}, Y_{i+1}) \subseteq E(G)$. Then by hypothesis we have that $T$ is a generalized rational graph $\lambda$-contraction and by Theorem \ref{main_theorem_3} it follows that $T$ has a fixed point.
\qed \end{proof}

As in \cite{branciari2002fixed, radenovic2010generalized}, let $\Gamma$ denote the set of all non-negative real-valued functions $\gamma$ with domain $\mathbb{R}_{+}$ such that
\begin{enumerate}[label=\textup{\alph*)}, ref=\alph*]
	\item $\gamma$ is Lebesgue integrable;
	\item $\gamma$ is a finite integral on every subset of $\mathbb{R}_{+}$ that is compact;
	\item $\int_{0}^{\epsilon} \gamma (t)dt>0$ for each $\epsilon > 0$.
\end{enumerate}

We obtain as a consequence of Theorem \ref{main_theorem_3} fixed point result satisfying integral type contraction conditions on a semi-metric space $(Y, \mathrm{d}_{s})$.

\begin{corollary}
	Let $(Y, \mathrm{d}_{s})$ be a $\mathrm{d}_{s}$-Cauchy complete semi-metric space endowed with a directed graph $G$, where $V(G) = Y$ and $\Delta \subseteq E(G)$. If the mapping $T: CB(Y) \rightarrow CB(Y)$ satisfies the following conditions:
	\begin{enumerate}[label=\textup{\roman*)}, ref=\roman*]
		\item An edge (path) between $U$ and $V$ implies an edge (path) between $T(U)$ and $T(V)$, for all $U, V \in CB(X)$.
		\item There exists $\alpha \in [0, 1)$, such that:
		\begin{align}
			\mathcal{H}_{s}(T(U), T(V)) \leq \int_{0}^{\alpha M_{T}(U, V)} \gamma(t) dt \label{semi_eqn04}
		\end{align}
		where, $M_{T}(U, V)$ is as in identity \eqref{semi_eqn01}.
	\end{enumerate}
	Then the following statements hold:
\begin{enumerate}[label=\textup{\arabic*)}, ref=\arabic*]
		\item If $F(T) \neq \emptyset$ is complete, then $PH$-weight assigned to $U, V \in F(T)$ is zero.
		\item If $F(T) \neq \emptyset$, then $Y_{T} \neq \emptyset$.
		\item If $Y_{T} \neq \emptyset$ and the weakly connected graph $\tilde{G}$ satisfies property $\mathcal{\overset{*}{P}}$, then $T$ has a unique fixed point.
		\item $F(T)$ is complete if and only if $F(T)$ is a singleton.
	\end{enumerate}
\end{corollary}

\begin{proof} \smartqed
	Define $\vartheta : [0, \infty) \rightarrow [0, \infty)$ as $\vartheta (x) = \int_{0}^{x} \gamma(t) dt$, then by (\ref{semi_eqn04}), we have
	\begin{align*}
		\mathcal{H}_{s}(T(U), T(V)) \leq \vartheta\Big( \alpha M_{T}(U, V) \Big),
	\end{align*}
which can be rewritten as
	\begin{align*}
		\mathcal{H}_{s}(T(U), T(V)) \leq \alpha \vartheta\Big(M_{T}(U, V) \Big).
	\end{align*}
Now, our result follows from applying Theorem \ref{main_theorem_3}.
\qed \end{proof}



\end{document}